\documentclass[10pt,a4paper,reqno]{amsart}
\usepackage{amsfonts,amsthm,latexsym,amsmath,amssymb,amscd,amsmath, epsf}

\usepackage[T2A]{fontenc}
\usepackage[cp1251]{inputenc}
\usepackage{latexsym,amssymb,amsthm}
\usepackage{url,graphics} 
\usepackage[dvips]{graphicx,epsfig} 

\newtheorem{theorem}{Theorem}

\newtheorem{lemma}[theorem]{Lemma}
\newtheorem{definition}{Definition}

\newtheorem{remark}{Remark}

\newtheorem{conjecture}{Conjecture}

\def\q#1.{{\bf #1.}}
\renewcommand\geq{\geqslant}
\renewcommand\leq{\leqslant}

\newcommand{\bR}{\mathbb{R}}

\newcommand{\be}{\begin{equation}}
\newcommand{\ee}{\end{equation}}

\newcommand {\ga}{\gamma}

\newcommand {\diff} {\mbox{\emph{diff}}} 

\begin{document}
          \numberwithin{equation}{section}

          \title[Counting the minimal number of inflections of a plane curve]
          {Counting the minimal number of inflections of a plane curve}

\author[G. Nenashev]{Gleb Nenashev}
\address{  Chebyshev Laboratory, St. Petersburg State University, 14th Line, 29b, Saint Petersburg, 199178 Russia.}
\thanks{The author is supported by Rokhlin grant, %by RF Government grant 11.G34.31.0026 
by the Chebyshev Laboratory  (Department of Mathematics and Mechanics, St. Petersburg State University)  under RF Government grant 11.G34.31.0026
and by JSC "Gazprom Neft".}
\email{glebnen@mail.ru}

\begin{abstract} 
Given a plane curve $\ga: S^1\to \bR^2$, we consider the problem of determining the minimal
number $I(\ga)$ of inflections which  curves $\diff(\ga)$ may have, where $\diff$
runs over the group of  diffeomorphisms of $\bR^2$. We show  that if $\ga$ is an immersed curve with $D(\ga)$ 
 double points and no other singularities,  
then $I(\ga)\leq 2D(\ga)$. In fact, we prove the latter result for the so-called plane doodles which are finite collections of closed immersed plane curves whose only singularities are double points.  
\end{abstract}
\maketitle

\section{\bf Introduction}

It is obvious that  any plane curve $\ga:S^1\to \bR^2$ diffeomorphic to the figure-eight must have at least two inflection points. 
Generalizing this observation, B.~Shapiro  posed in \cite{Sh} the problem of finding/estimating the minimal number 
of inflection points of a given immersed plane curve having only double points  
under the action of the group of
%Merx(
% plane diffeomorphisms. 
 diffeomorphisms of the plane. 
%Merx)
He obtained a number of results for the class of the so-called tree-like curves characterized by the 
property that  removal of any double point makes the curve disconnected.

\begin{definition}
A tree-like curve 
is a closed immersed 
plane curve with follow property: 
 removal of any double point with its neighborhood makes the curve disconnected
\end{definition}

In particular, using a natural  plane tree associated to any tree-like curve, he got lower and upper bounds for the number 
of inflections for such  curves and also found a criterion  
when a tree-like curve can be drawn without inflections.

 When we say that a curve  $\gamma$ can be drawn with a certain number 
of inflection points we mean that there is a plane diffeomorphism
$\diff$
 such that $\diff(\gamma)$ has 
 that many inflections. 
Respectively drawing is $\diff(\gamma)$.

In what follows we shall work with the following natural generalization of immersed plane curves with at most double points, comp. e.g. \cite{Me1}. 

\begin{definition} A doodle is a union of a finite number of closed immersed plane curves 
without triple intersections.
\end{definition}

The main result of this note is as follows. 

\begin{theorem}\label{th:main}
Any doodle with $n$ double points can be drawn with at most $2n$ inflection points. 
\end{theorem}

We conjecture the following stronger statement.

\begin{conjecture}\label{conj:main} 
Any closed plane curve with $n$ double points can be drawn with at most $n+1$ inflection points. 
\end{conjecture}

This conjecture is true for tree-like curves. %Slightly stronger theorem holds, which contains a tight estimate.

\begin{theorem}\label{th:tree}
Any tree-like curve with $n$ double points except figure-eight can be drawn with at most $n$ inflection points.
%i.e. $I(\gamma)\leq D(\gamma)$.
\end{theorem}

The bound from Theorem~\ref{th:tree} is tight. There are examples with $2k$ double points, which can not be drawn with less than $2k$ inflections. We must take the closed curve with alternating $2k$ loops by turn outward and inward.

In complement to Theorems~\ref{th:main} and~\ref{th:tree}, 
  we present  in \S3 an infinite family of topologically distinct minimal fragments forcing 
an inflection point which implies that the problem of defining the exact minimal 
number of inflection points of a given doodle is algorithmically very hard.  Therefore there is no chance to obtain an explicit formula for the latter number except for some very special families of plane curves. 
Our results seem to support the general principle that invariants of curves and knots of geometric origin 
are  difficult to calculate even algorithmically.  Observe that algebraic invariants of doodles similar 
to Vassiliev invariants of knots were introduced by V.~I.~Arnold in \cite{Ar} and 
later  considered by number of authors. See especially,  \cite{Me1}, \cite{Me2}. 

\medskip
\noindent
{\bf Acknowledgement.}  The author is  grateful to the Mathematics Department 
of Stockholm University for the financial support of his visit to 
Stockholm in November 2013 and to Professor B.~Shapiro for the formulation of the problem.

%%%%%%%%%%%%%%%%%%%%%%%%%%%%%%%%%%%%%%%%%%%%%%%%%%%%%%%%%%%%%%%%%%%%%%%%%%%%
\bigskip
\section{\bf Proofs}

\begin{proof}[Proof of Theorem~\ref{th:main}] Assume the contrary, i.e. that there exists a doodle with $n$ double points 
which can not be drawn with less than $2n+1$ inflections.
Let us consider a counterexample with the minimal number of double points. 

Obviously our counterexample is not an embedded circle and it is connected. 
Consider this doodle as an (obvious) planar graph $G$ 
with possible 
multiple 
edges and loops. 
Double points are the vertices  of this graph, and the arcs 
connecting double points are the edges.  

By faces of a doodle we mean the bounded faces of the (complement to the) planar graph. 
By the length of a face we mean  the number of edges 
 in its boundary.

\begin{lemma} 
A minimal counterexample has the following properties:

\begin{itemize}
	\item $a)$  there are no faces of length $1$.
	\item $b)$ there are no faces of length $2$.
	\item $c)$ there are no edges of multiplicity $\geq 3$.
\end{itemize}

\begin{proof}
{\bf a)}
 {\it Assuming that there exists a face of length $1$};  remove temporarily 
its boundary and remove the resulting vertex of valency $2$  
by gluing two edges into one. (It might happen that there will be
no vertices left.) Then we obtain a graph corresponding to a doodle 
with $n-1$ double points. 

Thus we can draw a new doodle with at most $2n-2$ inflection points. 
Then by returning back the removed face we add no more than $2$ inflection points, see Fig.~\ref{return-pet}.
Contradiction with the minimality assumption.

{\begin{figure}[htb!]

\centering
\includegraphics[scale=0.4]{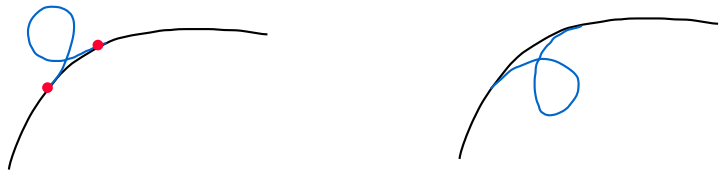}
\caption{Returning the face of length $1$.}
\label{return-pet}
\end{figure}
}

{\bf b)} {\it Assuming that there exists a face of length $2$}, denote the vertices of 
this face by $A, B$, and its edges by  $l_1, l_2$. Vertices $A$ and $B$ are distinct, since otherwise this common vertex 
would have valency $4$, and therefore there exist edges joining this face with other vertices. 
But then our doodle  has just one double point; it is easy to check that this can 
not be a counterexample. 

Remove edges $l_1,l_2$ and contract $A$ and $B$ to one vertex called  $\widehat{AB}$. We obtain a new 
doodle with  $n-1$ double points.   By the minimality of our counterexample we can draw it with no more than  $2n-2$  inflection points. 
Ungluing the double vertex $\widehat{AB}$  and smoothing the resulting picture we add exactly two new inflection points, see Fig.~\ref{return-kr}.
Contradiction with the minimality assumption.

{\begin{figure}[htb!]

\centering
\includegraphics[scale=0.55]{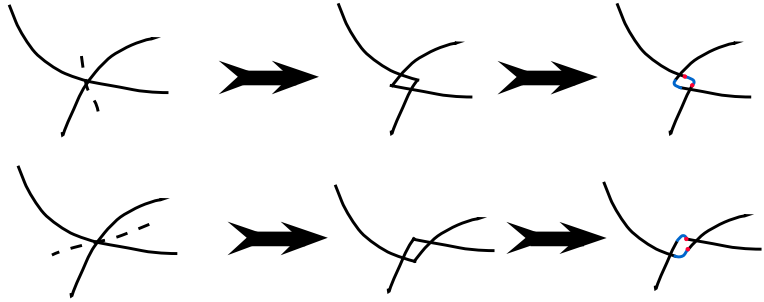}
\caption{Returning pairs of double edges bounding the face of length $2$.}
\label{return-kr}
\end{figure}
}

{\bf c)} {\it Assume that there exists a  triple edge}. 
Consider edges $l_1,l_2,l_3$, forming this triple edge and connecting a pair of vertices called $A,B$. 
(Observe  that $A$ and  $B$ are distinct, since otherwise their valency should be  $6$, but the maximal valency is  $4$.)

Edges  $l_1,l_2,l_3$ divide the plane in two finite domains and one infinite. Let us denote the finite domains by   $\sigma_1, \sigma_2$.

Both vertices $A$ and $B$  have exactly one additional  edge each. 
Either both these edges go inside $\sigma_i$ ($i=1,2$), or none of them goes inside $\sigma_i$. 
(Otherwise in the  graph induced by all vertices inside $\sigma_i$ one vertex will have valency $3$ and 
the remaining will have valency $4$, but the sum of all valencies must be even!). 
Thus edges can not go into $\sigma_1$ and  $\sigma_2$ simultaneously. 
 Without loss of generality assume that these edges do not go into $\sigma_1$. But then 
either the doodle is disconnected which is impossible, or 
 $\sigma_1$  is a domain with empty interior which is impossible by b).

\end{proof}
\end{lemma}

%%%%%%%%%%%%%%%%%%%%%%%%%%%%%%%%%%%%%%%%%%%%%%%%%%%%%%%%%%%%%%%%%%%%%%
Notice that in our doodle there still might be double edges or loops with non-empty interior.  
Let us  split each loop into three subedges by adding two fake vertices.
Additionally in each double edges we split one of them  into two subedges by adding one  fake vertex.

Denote by $G'$ the obtained planar graph; 
it does not contain multiple edges or loops. 
By Fary's theorem  it has a drawing $ \zeta $ in which all the edges are straight segments 
and  $ \zeta $ is equivalent to the original drawing.

Denote by $ \zeta '$ the drawing of the graph $G$ obtained by a smoothening of the angles between the edges at each vertex in the drawing $ \zeta $  (see Fig.~\ref{sgl}).

{\begin{figure}[htb!]

\centering
\includegraphics[scale=0.5]{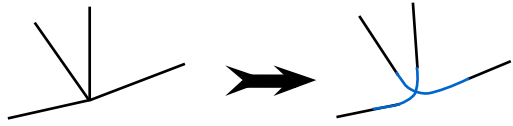}
\caption{Smoothing a vertex of valency $4$.}
\label{sgl}
\end{figure}
}

\begin{lemma}
In the drawing $ \zeta '$ each edge of the graph $ G $ contains at most one inflection.
\begin{proof}

If we do not split an edge, then obviously it has at most one inflection.

If an edge is split into three subedges, then it is a loop. Call it $ ABC $, where $ B, C $ are the fake vertices.
Vertex $A$ is the original and hence its valency is $4$. Thus it has exactly two other edges. 
 Either both other edges go inside the triangle  $\triangle ABC$ or both go outside this triangle.

If they go outside, then either this loop is a face of length $1$ or our doodle is disconnected.
Hence, both edges go inside $\triangle ABC $ (see Fig.~\ref{ABC}~$ I $) and then the loop $ABC$ has no inflections.

{\begin{figure}[htb!]
\centering
\includegraphics[scale=0.4]{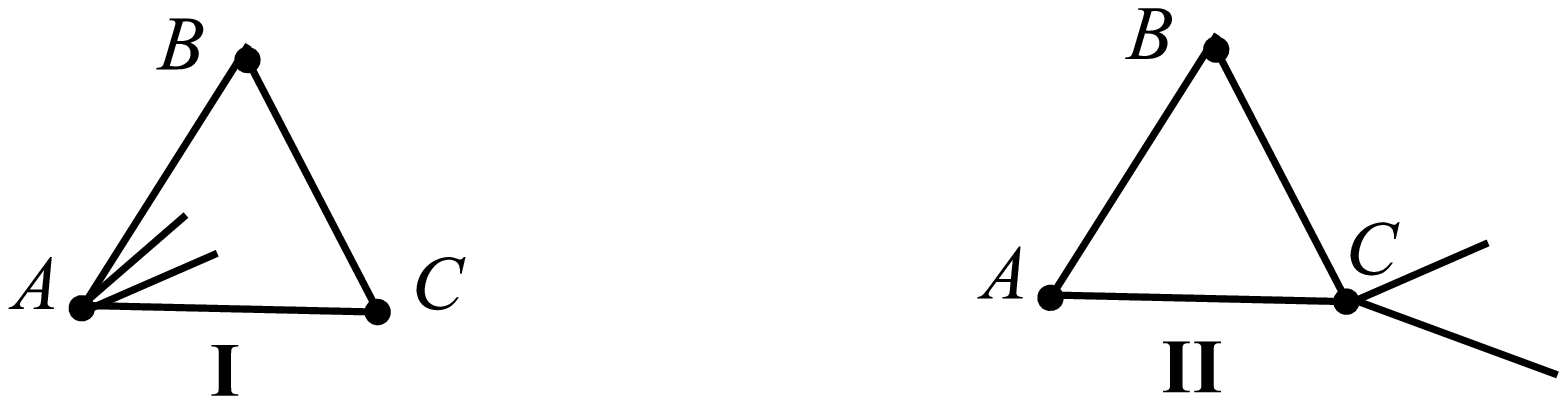}
\caption{}
\label{ABC}
\end{figure}
}

It remains to consider the case when the edge splits into two subedges. Call it $ ABC $ with the fake vertex $ B $. 
Then $AC$ is an edge in the graph $G$.
Consider the triangle $\triangle ABC $. 
If we go along the edge  $ABC$ across $C$ in the doodle we go inside a triangle or along the edge $CA$. 
Then  the part $BC$ of the edge $ABC$ has no inflection. Hence, there is at most one inflection on the edge $ABC$.
In the remaining case we go outside of the triangle; then the fourth edge of $ C $
also goes outside it (see Fig.~\ref{ABC}~$II$). 

Similarly, we need to consider the case when other edges of vertex $A$ go outside  $\triangle ABC$. 
Then these edges do not go inside  $\triangle ABC$, hence,  either there is a face of length 2 
or the doodle is disconnected. Both case are impossible.
Since we covered all possible cases, the lemma is proved.\end{proof} 
\end{lemma} 

Lemma~3 implies that the number of inflections does not exceed the number of edges, 
hence it is at most  $2n$. Theorem~1 is proved. \end{proof}

%%%%%%%%%%%%%%%%%%%%%%%%%%%%%%%%%%%%%%%%%%%%%%%%%%%%%%%%%%%%%%%%%%%%%%%%%%%%%%%%%%%%%%%%%%%%%%%%%%%%%%%%%%%%%%%%%%%%%%%%%%%%%%%%%%%%
\bigskip

\begin{proof}[Proof of Theorem~\ref{th:tree}]

We prove that minimal number of inflections is not more  than number of double points plus $1$.
The idea of the proof without plus $1$ can be found in the remark~\ref{m1}.

Now assume the contrary, i.e. that there exists a tree-like curve with $n$ double points 
which can not be drawn with less than $n+2$ inflections. 
Let us consider a counterexample with the minimal number of double points. Obviously, $n>1$.

%Consider our tree-like curve as appropriate tree.
Consider the tree that corresponds to our curve.
We split our curve into $n+1$ closed parts of curve, these parts corresponds to vertices of the tree 
and points of tangency of these parts corresponds to edges of the tree  (see fig.~\ref{tree}, more information about appropriate tree see in~\cite{Sh}).

{\begin{figure}[htb!]
\centering
\includegraphics[scale=0.4]{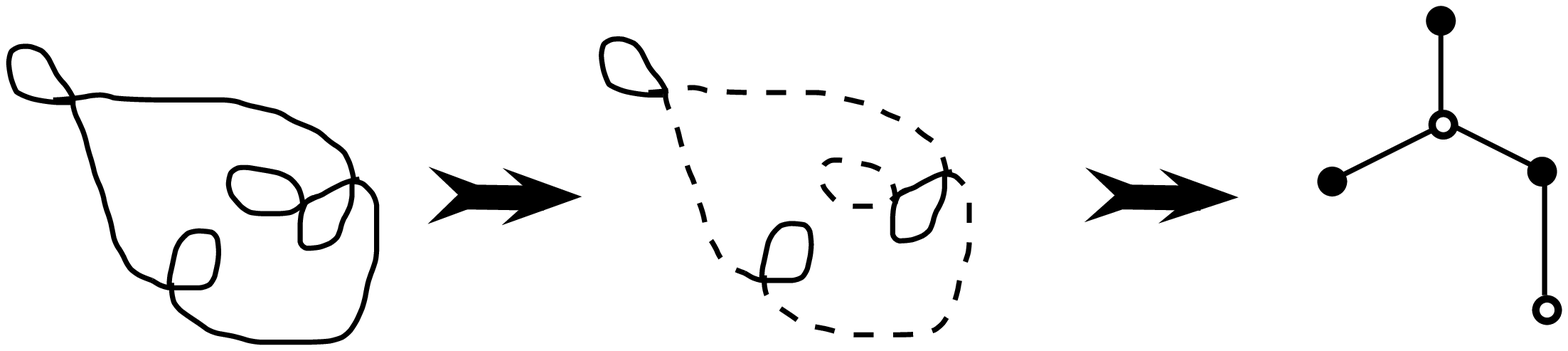}
\caption{}
\label{tree}
\end{figure}
}

%
%Perhaps outside face corresponds to vertex  then we call this vertex bad, resp. other vertices are good.
%Consider the longest path in this tree, denote it by $v_1\dots v_k$. Obviously,  $k>2$, otherwise $n=1$.
If the outer face corresponds to vertex of the tree, then we call this vertex {\it bad}, all other vertices are called {\it good}.
Let $v_1\dots v_k$ be the longest path in the tree.

\begin{lemma} For a minimal counterexample the following conditions are impossible.
\begin{itemize}
	\item $a)$  The vertex $v_2$ ($v_{k-1}$)  has degree $deg(v_2)>2$ and it is adjacent to $deg(v_2) - 1 $ good leaf vertices.
               
  \item $b)$ The vertex $v_2$ ($v_{k-1}$) has degree $deg(v_2)=2$, $v_1,v_2$ are good and $v_1\cup v_2$ is not boundary
 of outer face.
\end{itemize}
\begin{proof}

%{\bf a)} $deg(v_2)$ vertices attached  to the vertex $v_2$, of which $deg(v_2)-1$ are good leaf vertices. Consequently, there are two good leaf vertices which are attached in a sequence. 
%Remove these two leaf vertices, we got a smaller tree-like curve, hence, it is not a counterexample, then we can draw it and after return two leaf vertices with the addition of no more than two inflections (see fig.~\ref{del-2pet}, left 1-3).

{\bf a)} The vertex $v_2$ is adjacent to $deg(v_2)$ vertices and $deg(v_2)-1$ of them are good leaves. Consequently, there are two good leaves which are attached in a sequence. 
Removing these two leaves, we obtain a smaller tree-like curve, hence, it is not a counterexample.
We can draw this curve with number of inflections is less that number of double points plus $1$
 and after that we return two  deleted leaves with addition no more than two inflections (see fig.~\ref{del-2pet}, left 1-3).

{\begin{figure}[htb!]
\centering
\includegraphics[scale=0.2]{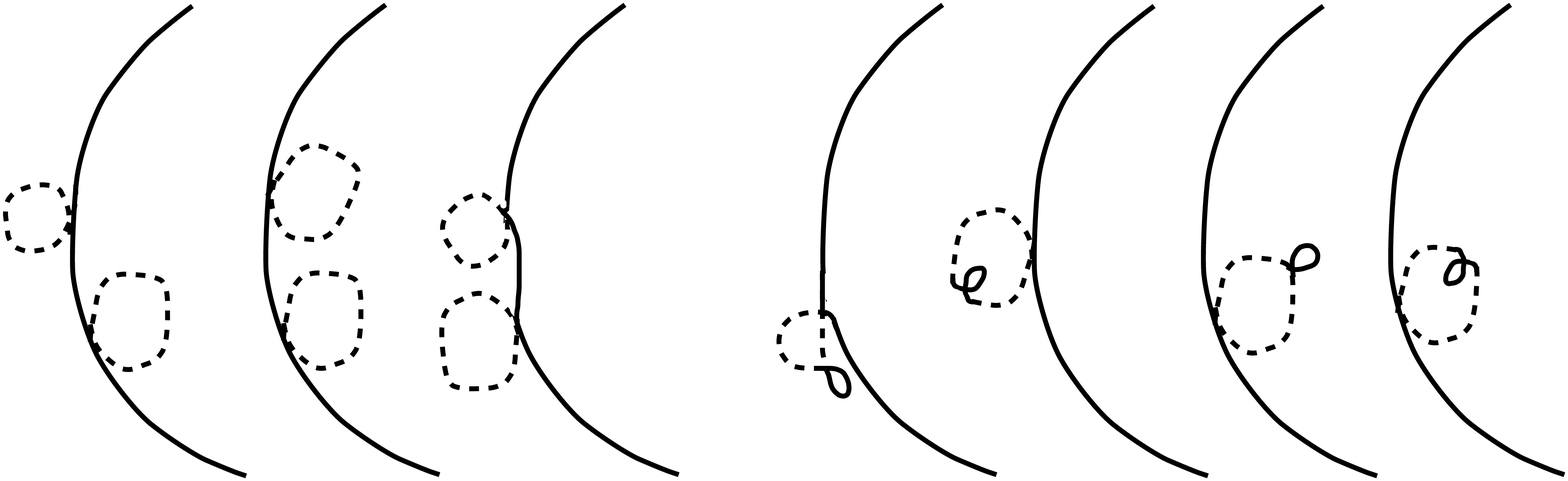}
\caption{}
\label{del-2pet}
\end{figure}
}

%{\bf b)} The vertex $v_2$ adjacent only to  vertices $v_1$ and $v_3$, furthermore $v_3$ is attached to outer side of $v_2$. 
%Remove two vertices $v_1$ and $v_2$, we got a smaller tree-like curve, hence, it is not a counterexample, then we can draw it and after return this two vertices with the addition of no more than two inflections (see fig.~\ref{del-2pet}, 4-7).

{\bf b)} The vertex $v_2$ is adjacent only to vertices $v_1$ and $v_3$, furthermore $v_3$ is attached to outer side of $v_2$. 
Removing vertices $v_1$ and $v_2$, we obtain a smaller tree-like curve, hence, it is not a counterexample.
We can draw this curve with number of inflections is less that number of double points plus $1$
 and after that we return two deleted vertices with addition no more than two inflections (see fig.~\ref{del-2pet}, 4-7).

\end{proof}
\end{lemma}

Now return to the proof of our theorem. Consider the next case, let $deg(v_2)>2$.
The vertex $v_2$ is adjacent to at least $deg(v_2)-1$ leaves. Hence (by $a)$), one of these leaves is bad and $k>3$ (otherwise $v_2$ is adjacent to $deg(v_2)$ leaves). 
Then $v_{k-1}$,~$v_{k}$ are good and $v_{k-1}\cup v_{k}$ is not a boundary  of outer face, hence, $deg(v_{k-1})>2$ (otherwise we have a contradiction to $b)$). 
Similarly, $v_ {k-1}$ is adjacent to the bad vertex too. Then the bad vertex has degree at least two, but it is a leaf in this case. Hence, this case is not possible.

Then $v_2$ has degree $2$ and, analogically, the vertex $v_{k-1}$ has degree $2$. Furthermore, $v_1$, $v_2$ or $v_1\cup v_2$ is the boundary of outer face (otherwise we have a contradiction to $b)$).
Similarly, $v_{k-1}$, $v_{k}$ or $v_{k-1}\cup v_{k}$ is the boundary of outer face. Hence, $k=3$ and $v_2$ is a bad vertex. Then all vertices except $v_2$ are attached to the inner side of $v_2$, but this tree-like curve can be drawn without inflections. This is a contradiction. We consider all possible cases, the theorem is proved.

\end{proof}

\begin{remark}
\label{m1}
To prove the bound without plus $1$ we must prove that tree-like curves with $3$ double points are not counterexamples, because our proof is based on the step from $n$ to $n-2$.

%This estimate is tight.
%There are examples with $2k$ double points, which can not be drawn with less than $2k$ double.
%We must take the closed curve with alternating $2k$ loops by turn outward and inward (proof that at least $2k$ inflections see in~\cite{Sh}).
\end{remark}

\bigskip
%%%%%%%%%%%%%%%%%%%%%%%%%%%%%%%%%%%%%%%%%%%%%%%%%%%%%%%%%%%%%%%%%%%%%%%%%%%%%%%%%%%%%%%%%%%%%%%%%%%%%%%%%%%%%%%%%%%%%%%%%%%%%%%%%%%%
\section{\bf On minimal fragments forcing an inflection.}

\begin{definition} A {\it fragment} is the union of a finite number of immersed plane curves without triple intersections
(up to diffeomorphisms).
\end{definition}

Obviously, if a doodle $\ga$ has $k$ disjoint
fragments forcing an inflection (see next definition), then any drawing of $\ga$
contains at least $k$ inflections.

\begin{definition} A fragment is called {\it a minimal fragment forcing an inflection} if the following two conditions are satisfied (see Fig.~\ref{min}):

\begin{itemize}
	\item any drawing of this fragment necessarily contains an inflection point.

	\item removing any double point or any curve or cutting any curve (between two double points) we obtain a fragment which can be drawn without inflection points. 
\end{itemize}

\end{definition}

{\begin{figure}[htb!]
\centering
\includegraphics[scale=0.5]{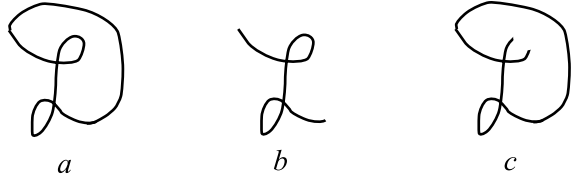}
\caption{Fragments forcing an inflection point. $a$ -- non-minimal, $b,c$ -- minimal.}
\label{min}
\end{figure}
}

\begin{remark}Obviously, any 
minimal fragment forcing an inflection
is
connected.
\end{remark}

In this section we construct an infinite series of minimal fragments forcing an inflection. Additionally,
this construction implies the following result:

\begin{theorem}
\label{exp}
There exists $c>0$ such that the number of fragments 
forcing an inflection with at most $n$ double points is at least $e^{cn}$. 
\end{theorem}

The above theorem is true even for fragments consisting of curves without self-intersections, but for $n$ at least some $N_0$.
That fact in its turn makes it very hard not only to count the minimum number of inflections of a given doodle but also to find a 
criterion when a doodle can be drawn without inflections.
Now let us construct a series of minimal fragments.

\begin{definition}
{\it A key}  $b$ for the curve $z$ is a curve shown in Fig.~\ref{key}.
\end{definition}

{\begin{figure}[htb!]
\centering
\includegraphics[scale=0.5]{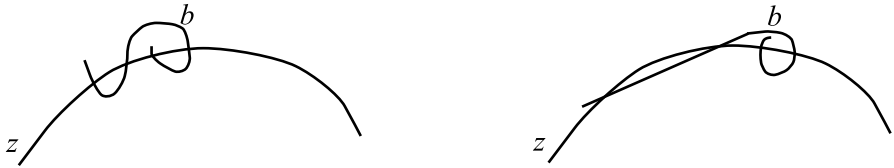}
\caption{Key $b$ for a curve $z$.}
\label{key}
\end{figure}
}

\begin{lemma} 

\begin{enumerate}
	\item If a drawing of a curve $z$  has no inflections, 
	then its key  determines the direction of convexity of the curve $z$.
  \item If a curve $z$ is convex in the right direction, then its key can always be drawn without inflections.
	\item If a part of key is removed, 
	then the remaining parts of the key can always be drawn without inflections.
\end{enumerate}

\begin{proof}
Items $1$ and $3$ are obvious. In the right-hand of Fig.~\ref{key} it is shown how to draw  a key in item $2$.

\end{proof}
\end{lemma}

%%%%%%%%%%%%%%%%%%%%%%%%%%%%%%%%%%%%%%%%%%%%%%%%%%%%%%%%%%%%%%%%%%%%%%%%%%%%%%%%%%%%%%%%%%%%%%%%%%%%%%%%%%%%%%%%%%%%%%%%%%%%%%%%%%%%

Now we present an infinite series of distinct minimal fragments forcing an inflection. %A series  $\Omega$  of fragments 
It consists of fragments  having the following form:
\begin{itemize}
	\item $k\geq 3$ curves bound a domain in which each curve intersects only with its neighbors and goes after crossing inside the domain 
	(see Fig.~\ref{ex}, left).
	\item Each of these curves has either a key of type $II$ or $III$ or a loop close to one of its endpoints (see Fig.~\ref{ex}).
\end{itemize}

{\begin{figure}[htb!]
\centering
\includegraphics[scale=0.4]{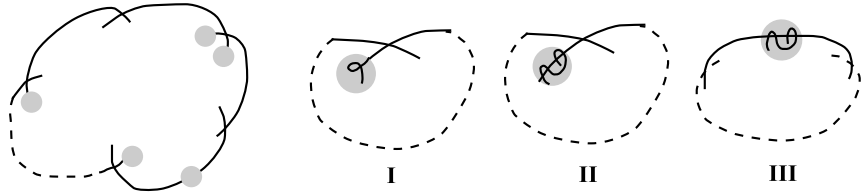}
\caption{Minimal fragments forcing  inflection points.}
\label{ex}
\end{figure}
}

\begin{theorem}
\label{omega}
Fragments in the above series are minimal fragments forcing  inflections.
\begin{proof}

Consider a fragment consisting of $k$ curves (excluding keys).
 This fragment must contain an inflection point, because otherwise  
all curves are convex 
%Merx(
%in the direction of the center of the "`k-gon"
inwards
%Merx)
 (due to the presence of keys or loops)
and the "vertices" of the $k$-gon have the same convexity, but this is impossible.

%%%%%%%%%%%%%%%%%%%%%%%%%%%%%%%%%%%%%%%%%%%%%%%%%%%%%%%%%%%%%%%%%%%%%%%%%%%%%%%%%dalee

It remains to prove that this fragment is minimal.
%{1^\circ} {\it If we remove something from at least one key or a loop} %or a part of curve outside key
%%or part of curve inside the key type~$II$
%then we  can draw one curve convex outward, and the others draw inside ( see Fig.~\ref{ex-del}, left). After that we can draw all loops and keys.
%
%If we remove part of curve inside key type~$III$ 
%then we  can draw the part of curve convex outward, and the other part of this curve  and other curves draw inside and the key ( see Fig.~\ref{ex-del}, middle and right). 
%After that we can draw all loops and keys.
%

{$1^\circ$} {\it If we remove something from at least one key or a loop or cut a loop or a key of type~$II$.}
Then we  can draw one curve convex 
%Merx(
%in the outside direction of the "k-gon", 
outwards
%Merx)
and 
%Merx(
%draw the others in the inside direction 
the others convex inwards %repeat "convex" instead of "draw"
%Merx)
(see Fig.~\ref{ex-del}, left). 
After that we can draw all loops and other keys without inflections.

{$2^\circ$} {\it If we remove a part of a curve inside a key of type~$III$.}
Then we  can draw the part of a curve  convex 
%Merx(
%in outside direction of the "k-gon", and 
outwards,
%Merx)
the other part of this curve  and other curves 
%Merx(
%draw in inside direction, 
convex inwards,
%Merx)
and 
%Merx(
%draw 
%Merx)
this key of type~$III$ 
without inflections (see Fig.~\ref{ex-del}, middle and right). 
After that we can draw all loops and keys without inflections.

{$3^\circ$} {\it If we remove a part of curve inside a key of type~$II$.}
This case can be proved
by combining the ideas of cases $1^\circ$ and $2^\circ$. 
We draw "big" part of the curve convex outwards and other $k-1$ curves inwards (see Fig.~\ref{ex-del}, left)
and later we draw the key of type~$II$ with second part of this curve on the end of "big" part (see Fig.~\ref{ex-del}, right). 
After that we can draw all loops and keys without inflections.

{$4^\circ$} {\it If we cut a curve in the boundary of k-gon outside a key of type~$III$.}
This case
is obvious. We can do all $k$  curves with convexity in the correct direction, because
we should not  build a "$k$-gon".

{\begin{figure}[htb!]
\centering
\includegraphics[scale=0.3]{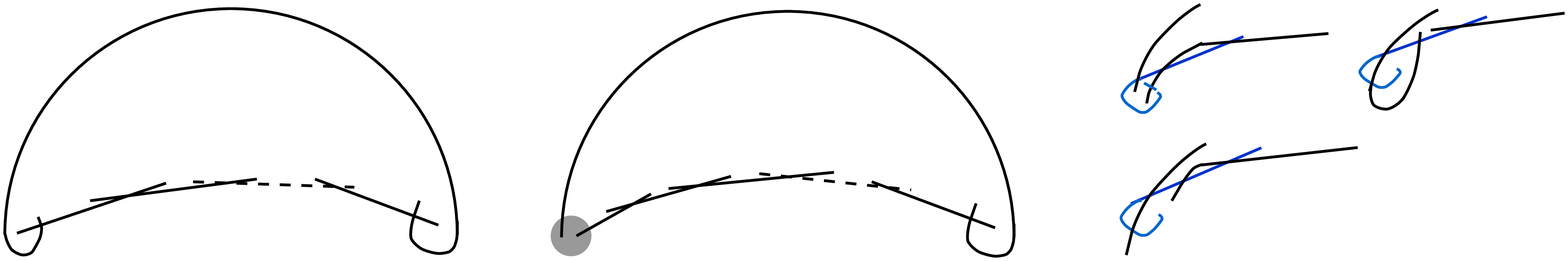}
\caption{}
\label{ex-del}
\end{figure}
}

We 
have
considered all possible cases, 
so
the theorem is proved.
\end{proof}
\end{theorem}

\begin{proof}[Proof of Theorem~\ref{exp}]

We will use only loops (similarly, we could use only keys of type $II$ and $III$). Fixing $k>0$, we have 2 possibilities  for each loop.
Hence, we have at least $2^k/k$ minimal fragments, because each fragment is considered at most $k$ times. 
They have exactly $2k$ double points.
Now it is obvious that there exists desired $c>0$. 
\end{proof}

\end{document}